\documentclass[11pt,reqno]{amsart}

\usepackage{caption,amsmath,amssymb,amsfonts,amsthm,latexsym,graphicx,multirow,color,hyperref,enumerate}
\usepackage[all]{xy}

\numberwithin{table}{section}

\newcommand{\A}{\mathrm{A}}    \newcommand{\Aut}{\mathrm{Aut}}
 \newcommand{\bbF}{\mathbb{F}} 
         
\newcommand{\D}{\mathrm{D}}

\newcommand{\G}{\mathrm{G}} \newcommand{\GaG}{\mathrm{\Gamma G}} \newcommand{\GaL}{\mathrm{\Gamma L}} \newcommand{\GaSp}{\mathrm{\Gamma Sp}}    \newcommand{\GL}{\mathrm{GL}}

\newcommand{\lefthat}{\scalebox{1.3}[1]{\text{$\hat{~}$}}}
\newcommand{\M}{\mathrm{M}} \newcommand{\magma}{\textsc{Magma}}
\newcommand{\N}{\mathrm{N}}

\newcommand{\Out}{\mathrm{Out}}
  \newcommand{\PGL}{\mathrm{PGL}} \newcommand{\PGaL}{\mathrm{P\Gamma L}}      \newcommand{\PSL}{\mathrm{PSL}} \newcommand{\PSiL}{\mathrm{P\Sigma L}}    \newcommand{\PSU}{\mathrm{PSU}}

\newcommand{\SiL}{\mathrm{\Sigma L}}  \newcommand{\SL}{\mathrm{SL}}  \newcommand{\Soc}{\mathrm{Soc}} \newcommand{\Sp}{\mathrm{Sp}}  \newcommand{\SU}{\mathrm{SU}}  \newcommand{\Sy}{\mathrm{S}}

\newtheorem{theorem}{Theorem}[section]
\newtheorem{lemma}[theorem]{Lemma}

\newtheorem{problem}[theorem]{Problem}

\theoremstyle{definition}

\newtheorem*{remark}{Remark}

\begin{document}

\title[Factorizations of almost simple linear groups]{Factorizations of almost simple linear groups}

\author[Li]{Cai Heng Li}
\address{(Li) SUSTech International Center for Mathematics and Department of Mathematics\\Southern University of Science and Technology\\Shenzhen 518055\\Guangdong\\P.~R.~China}
\email{lich@sustech.edu.cn}

\author[Wang]{Lei Wang}
\address{(Wang) School of Mathematics and Statistics\\Yunnan University\\Kunming 650091\\Yunnan\\P.~R.~China}
\email{wanglei@ynu.edu.cn}

\author[Xia]{Binzhou Xia}
\address{(Xia) School of Mathematics and Statistics\\The University of Melbourne\\Parkville 3010\\VIC\\Australia}
\email{binzhoux@unimelb.edu.au}

\begin{abstract}
This is the first one in a series of papers classifying the factorizations of almost simple groups with nonsolvable factors.
In this paper we deal with almost simple linear groups.


\end{abstract}

\maketitle

\section{Introduction}

An expression $G=HK$ of a group $G$ as the product of subgroups $H$ and $K$ is called a \emph{factorization} of $G$, where $H$ and $K$ are called \emph{factors}. A group $G$ is said to be \emph{almost simple} if $S\leqslant G\leqslant\Aut(S)$ for some nonabelian simple group $S$, where $S=\Soc(G)$ is the \emph{socle} of $G$. In this paper, by a factorization of an almost simple group we mean that none its factors contains the socle. The main aim of this paper is to solve the long-standing open problem:

\begin{problem}\label{PrbXia1}
Classify factorizations of finite almost simple groups.
\end{problem}

Determining all factorizations of almost simple groups is a fundamental problem in the theory of simple groups, which was proposed by Wielandt~\cite[6(e)]{Wielandt1979} in 1979 at The Santa Cruz Conference on Finite Groups. It also has numerous applications to other branches of mathematics such as combinatorics and number theory, and has attracted considerable attention in the literature.
In what follows, all groups are assumed to be finite if there is no special instruction.

The factorizations of almost simple groups of exceptional Lie type were classified by Hering, Liebeck and Saxl~\cite{HLS1987}\footnote{In part~(b) of Theorem~2 in~\cite{HLS1987}, $A_0$ can also be $\G_2(2)$, $\SU_3(3)\times2$, $\SL_3(4).2$ or $\SL_3(4).2^2$ besides $\G_2(2)\times2$.} in 1987.
For the other families of almost simple groups, a landmark result was achieved by Liebeck, Praeger and Saxl~\cite{LPS1990} thirty years ago, which classifies the maximal factorizations of almost simple groups. (A factorization is said to be \emph{maximal} if both the factors are maximal subgroups.)
Then factorizations of alternating and symmetric groups are classified in~\cite{LPS1990}, and factorizations of sporadic almost simple groups are classified in~\cite{Giudici2006}.
This reduces Problem~\ref{PrbXia1} to the problem on classical groups of Lie type.
Recently, factorizations of almost simple groups with a factor having at least two nonsolvable composition factors are classified in~\cite{LX2019}\footnote{In Table~1 of~\cite{LX2019}, the triple $(L,H\cap L,K\cap L)=(\Sp_{6}(4),(\Sp_2(4)\times\Sp_{2}(16)).2,\G_2(4))$ is missing, and for the first two rows $R.2$ should be $R.P$ with $P\leqslant2$.}, and those with a factor being solvable are described in~\cite{LX} and~\cite{BL}.

As usual, for a finite group $G$, we denote by $G^{(\infty)}$ the smallest normal subgroup of $X$ such that $G/G^{(\infty)}$ is solvable.
For factorizations $G=HK$ with nonsolvable factors $H$ and $K$ such that $L=\Soc(G)$ is a classical group of Lie type, the triple $(L,H^{(\infty)},K^{(\infty)})$ is described in~\cite{LWX}. Based on this work, in the present paper we characterize the triples $(G,H,K)$ such that $G=HK$ with $H$ and $K$ nonsolvable, and $G$ is a linear group.
As important special cases, groups that are transitive on the set of $1$-spaces and antiflags, respectively, are classified in the literature (see~\cite{CK1979,Kantor,Liebeck1987}).

For groups $H,K,X,Y$, we say that $(H,K)$ contains $(X,Y)$ if $H\geqslant X$ and $K\geqslant Y$, and that $(H,K)$ \emph{tightly contains} $(X,Y)$ if in addition $H^{(\infty)}=X^{(\infty)}$ and $K^{(\infty)}=Y^{(\infty)}$. 
Our main result is the following Theorem~\ref{ThmLinear}. 
Note that it is elementary to determine the factorizations of $G/L$ as this group has relatively simple structure (and in particular is solvable).


\begin{theorem}\label{ThmLinear}
Let $G$ be an almost simple group with socle $L=\PSL_n(q)$, where $n\geqslant2$ and $(n,q)\neq(2,2)$ or $(2,3)$, and let $H$ and $K$ be nonsolvable subgroups of $G$ not containing $L$. Then $G=HK$ if and only if (with $H$ and $K$ possibly interchanged) $G/L=(HL/L)(KL/L)$ and $(H,K)$ tightly contains $(X^\alpha,Y^\alpha)$ for some $(X,Y)$ in Table~$\ref{TabLinear}$ and $\alpha\in\Aut(L)$.
\end{theorem}

\begin{remark}
Here are some remarks on Table~\ref{TabLinear}:
\begin{enumerate}[{\rm(I)}]
\item The column $Z$ gives the smallest almost simple group with socle $L$ that contains $X$ and $Y$. In other words, $Z=\langle L,X,Y\rangle$.
It turns out that $Z=XY$ for all pairs $(X,Y)$.
\item The groups $X$, $Y$ and $Z$ are described in the corresponding lemmas whose labels are displayed in the last column. 
\item The description of groups $X$ and $Y$ are up to conjugations in $Z$ (see Lemma~\ref{LemXia04}(b) and Lemma~\ref{LemXia03}).
\end{enumerate}
\end{remark}

\begin{table}[htbp]
\captionsetup{justification=centering}
\caption{$(X,Y)$ for linear groups}\label{TabLinear}
\begin{tabular}{|l|l|l|l|l|l|l|}
\hline
Row & $Z$ & $X$ & $Y$ & Remarks & Lemma\\
\hline
1 & $\PSL_n(q)$ & $\lefthat\SL_a(q^b)$, $\lefthat\Sp_a(q^b)'$ & $q^{n-1}{:}\SL_{n-1}(q)$ & $n=ab$ & \ref{LemLinear01}, \ref{LemLinear02}\\
2 & $\PSL_n(q)$ & $\G_2(q^b)'$ & $q^{n-1}{:}\SL_{n-1}(q)$ & $n=6b$, & \ref{LemLinear04}\\
  &   &   &   &  $q$ even &  \\
\hline
3 & $\PSL_n(q)$ & $\lefthat\Sp_n(q)'$ & $\SL_{n-1}(q)$ & & \ref{LemLinear05}\\
4 & $\SL_{2m}(2)$ & $\SiL_m(4)$, $\GaSp_m(4)$ & $\SL_{2m-1}(2)$ & & \ref{LemLinear06}, \ref{LemLinear07}\\
5 & $\SL_{2m}(2).2$ & $\SL_m(4).2$, $\Sp_m(4).2$ & $\SL_{2m-1}(2).2$ & & \ref{LemLinear16}, \ref{LemLinear17}\\
6 & $\PSiL_{2m}(4)$ & $\SiL_m(16)/d$, $\GaSp_m(16)$ & $\SiL_{2m-1}(4)$ & $d=(m,3)$ & \ref{LemLinear06}, \ref{LemLinear07}\\
7 & $\PSL_{2m}(4).2$ & $(\SL_m(16).4)/d$, $\Sp_m(16).4$ & $\SL_{2m-1}(4).2$ & $d=(m,3)$ & \ref{LemLinear16}, \ref{LemLinear17}\\
8 & $\PSL_6(q)$ & $\G_2(q)$ & $\SL_5(q)$ & $q$ even & \ref{LemLinear09}\\
\hline
9 & $\PSL_2(9)$ & $\PSL_2(5)$ & $\A_5$ & & \ref{LemLinear10}\\
10 & $\PSL_3(4).2$ & $\PGL_2(7)$ & $\M_{10}$ & & \ref{LemLinear11}\\
11 & $\SL_4(2)$ & $\SL_3(2)$, $2^3{:}\SL_3(2)$ & $\A_7$ & & \ref{LemLinear12}\\
12 & $\PSL_4(3)$ & $\Sy_5$, $4\times\A_5$, $2^4{:}\A_5$ & $3^3{:}\SL_3(3)$ & & \ref{LemLinear13}\\
13 & $\PSL_6(3)$ & $\PSL_2(13)$ & $3^5{:}\SL_5(3)$ & & \ref{LemLinear14}\\
14 & $\SL_{12}(2)$ & $\G_2(4).2$ & $\SL_{11}(2)$ & & \ref{LemLinear15}, \ref{LemLinear18}\\
15 & $\PSiL_{12}(4)$ & $\G_2(16).4$ & $\SL_{11}(4).2$ & & \ref{LemLinear15}, \ref{LemLinear18}\\
\hline
\end{tabular}
\vspace{3mm}
\end{table}

\section{Preliminaries}

In this section we collect some elementary facts regarding group factorizations.

\begin{lemma}\label{LemXia01}
Let $G$ be a group, let $H$ and $K$ be subgroups of $G$, and let $N$ be a normal subgroup of $G$. Then $G=HK$ if and only if $HK\supseteq N$ and $G/N=(HN/N)(KN/N)$.
\end{lemma}

\begin{proof}
If $G=HK$, then $HK\supseteq N$, and taking the quotient modulo $N$ we obtain
\[
G/N=(HN/N)(KN/N).
\]
Conversely, suppose that $HK\supseteq N$ and $G/N=(HN/N)(KN/N)$. Then 
\[
G=(HN)(KN)=HNK 
\]
as $N$ is normal in $G$. Since $N\subseteq HK$, it follows that $G=HNK\subseteq H(HK)K=HK$, which implies $G=HK$.
\end{proof}

Let $L$ be a nonabelian simple group. We say that $(H,K)$ is a \emph{factor pair} of $L$ if $H$ and $K$ are subgroups of $\Aut(L)$ such that $HK\supseteq L$. For an almost simple group $G$ with socle $L$ and subgroups $H$ and $K$ of $G$, Lemma~\ref{LemXia01} shows that $G=HK$ if and only if $G/L=(HL/L)(KL/L)$ and $(H,K)$ is a factor pair.
As the group $G/L$ has a simple structure (and in particular is solvable), it is elementary to determine the factorizations of $G/L$.
Thus to know all the factorizations of $G$ is to know all the factor pairs of $L$.
Note that, if $(H,K)$ is a factor pair of $L$, then any pair of subgroups of $\Aut(L)$ containing $(H,K)$ is also a factor pair of $L$.
Hence we have the following:

\begin{lemma}\label{LemXia02}
Let $G$ be an almost simple group with socle $L$, and let $H$ and $K$ be subgroups of $G$ such that $(H,K)$ contains some factor pair of $L$. Then $G=HK$ if and only if $G/L=(HL/L)(KL/L)$.
\end{lemma}

In light of Lemma~\ref{LemXia02}, the key to determine the factorizations of $G$ with nonsolvable factors is to determine the minimal ones (with respect to the containment) among factor pairs of $L$ with nonsolvable subgroups.

\begin{lemma}\label{LemXia03}
Let $L$ be a nonabelian simple group, and let $(H,K)$ be a factor pair of $L$.
Then $(H^\alpha,K^\alpha)$ and $(H^x,K^y)$ are factor pairs of $L$ for all $\alpha\in\Aut(L)$ and $x,y\in L$.
\end{lemma}

\begin{proof}
It is evident that $H^\alpha K^\alpha=(HK)^\alpha\supseteq L^\alpha=L$. Hence $(H^\alpha,K^\alpha)$ is a factor pair.
Since $xy^{-1}\in L\subseteq HK$, there exist $h\in H$ and $k\in K$ such that $xy^{-1}=hk$. Therefore, 
\[
H^xK^y=x^{-1}Hxy^{-1}Ky=x^{-1}HhkKy=x^{-1}HKy\supseteq x^{-1}Ly=L,
\]
which means that $(H^x,K^y)$ is a factor pair.
\end{proof}

The next lemma is~\cite[Lemma~2(i)]{LPS1996}.

\begin{lemma}\label{LemXia05}
Let $G$ be an almost simple group with socle $L$, and let $H$ and $K$ be subgroups of $G$ not containing $L$. If $G=HK$, then $HL\cap KL=(H\cap KL)(K\cap HL)$. 
\end{lemma}

The following lemma implies that we may consider specific representatives of a conjugacy class of subgroups when studying factorizations of a group.

\begin{lemma}\label{LemXia04}
Let $G=HK$ be a factorization. Then for all $x,y\in G$ we have $G=H^xK^y$ with $H^x\cap K^y\cong H\cap K$.
\end{lemma}
  
\begin{proof}
As $xy^{-1}\in G=HK$, there exists $h\in H$ and $k\in K$ such that $xy^{-1}=hk$. Thus
\[
H^xK^y=x^{-1}Hxy^{-1}Ky=x^{-1}HhkKy=x^{-1}HKy=x^{-1}Gy=G,
\]
and
\[
H^x\cap K^y=(H^{xy^{-1}}\cap K)^y\cong H^{xy^{-1}}\cap K=H^{hk}\cap K=H^k\cap K=(H\cap K)^k\cong H\cap K.\qedhere
\]
\end{proof}

\section{Notation}

Throughout this paper, let $q=p^f$ be a power of a prime $p$, let $n\geqslant2$ be an integer such that $(n,q)\neq(2,2)$ or $(2,3)$, let $\,\overline{\phantom{\varphi}}\,$ be the homomorphism from $\GaL_n(q)$ to $\PGaL_n(q)$ modulo scalars, let $V$ be a vector space of dimension $n$ over $\bbF_q$, let
\[
v\in V\setminus\{0\},
\]
let $W$ be a hyperplane of $V$ not containing $v$, let $\phi$ be a field automorphism of $L$ of order $f$, and let $\gamma$ be the graph automorphism of $L$. 
Then $\phi$ and $\gamma$ commute, and so $|\phi\gamma|$ is the least common multiple of $f$ and $2$.
By abuse of notation, we also let $\phi$ and $\gamma$ denote the corresponding elements in $\Aut(\SL_n(q))$ and $\Out(L)$.
Recall that for each $g\in\SL(V)$, if we identify $V^*$ with $V$ in the canonical way, then $g^\gamma$ is the corresponding linear transformation of $(g^*)^{-1}$ on $V^*$, where $V^*$ is the dual space of $V$ and 
\[
g^*\colon V^*\to V^*,\quad\varphi\mapsto g\varphi
\]
is the \emph{pullback} of $g$.

If $n=2m$ is even, then the vector space $V$ can be regarded alternatively as a vector space $V_\sharp$ of dimension $m$ over $\bbF_{q^2}$. In this case, let $v_1,\dots,v_m$ be a basis of $V_\sharp$, let $\psi\in\SiL(V_\sharp)$ such that
\[
\psi\colon a_1v_1+\dots+a_mv_m\mapsto a_1^pv_1+\dots+a_m^pv_m
\]
for $a_1,\dots,a_m\in\bbF_{q^2}$, and let $\lambda$ be a generator of $\bbF_{q^2}^\times$. Then $v_1,\lambda v_1,\dots,v_m,\lambda v_m$ is a basis of $V$. 
Notice that a pullback of a linear transformation on $V_\sharp$ is a linear transformation on the dual space of $V_\sharp$. 
Thus $\gamma$ normalizes $\SL(V_\sharp)$.

\section{Infinite families of $(X,Y)$ in Table~\ref{TabLinear}}\label{SecLinear01}

In the first two lemmas we construct the factor pairs $(X,Y)$ in Row~1 of Table~\ref{TabLinear}.


\begin{lemma}\label{LemLinear01}
Let $G=\SL(V)=\SL_n(q)$, let $H=\SL_a(q^b)<G$ with $ab=n$, let $K=G_v$, let $Z=\overline{G}$, let $X=\overline{H}$, and let $Y=\overline{K}$. 
Then $H\cap K=q^{n-b}{:}\SL_{a-1}(q^b)$, and $Z=XY$ with $Z=\PSL_n(q)$, $X=\lefthat\SL_a(q^b)$ and $Y\cong K=q^{n-1}{:}\SL_{n-1}(q)$.
\end{lemma}

\begin{proof} 
It is clear that $Z=\PSL_n(q)$, $X=\lefthat\SL_a(q^b)$, and $Y\cong K=q^{n-1}{:}\SL_{n-1}(q)$. Since
\[
H\cap K=H\cap G_v=H_v=(q^b)^{a-1}{:}\SL_{a-1}(q^b)=q^{n-b}{:}\SL_{a-1}(q^b),
\]
we obtain
\[
\frac{|G|}{|K|}=\frac{|\SL_n(q)|}{|q^{n-1}{:}\SL_{n-1}(q)|}=q^n-1=\frac{|\SL_a(q^b)|}{|q^{n-b}{:}\SL_{a-1}(q^b)|}=\frac{|H|}{|H\cap K|},
\]
and so $G=HK$. This implies that $Z=\overline{G}=\overline{H}\,\overline{K}=XY$.
\end{proof}

\begin{lemma}\label{LemLinear02}
Let $G=\SL(V)=\SL_n(q)$ with $n$ even, let $H=\Sp_a(q^b)'<G$ with $ab=n$ and $a$ even, let $K=G_v$, let $Z=\overline{G}$, let $X=\overline{H}$, and let $Y=\overline{K}$. Then 
\[
H\cap K=
\begin{cases}
2^2{:}\Sp_2(2)&\textup{if }(a,b,q)=(4,1,2)\\
[q^{n-b}]{:}\Sp_{a-2}(q^b)&\textup{if }(a,b,q)\neq(4,1,2),
\end{cases}
\]
and $Z=XY$ with $Z=\PSL_n(q)$, $X=\lefthat\Sp_a(q^b)'$ and $Y\cong K=q^{n-1}{:}\SL_{n-1}(q)$.
\end{lemma}

\begin{proof}
It is clear that $Z=\PSL_n(q)$, $X=\lefthat\Sp_a(q^b)'$, and $Y\cong K=q^{n-1}{:}\SL_{n-1}(q)$. 
For $(a,b,q)=(4,1,2)$, we have $H=\Sp_4(2)'\cong\A_6$, and computation in \magma~\cite{BCP1997} shows that $H\cap K=2^2{:}\Sp_2(2)$ and $Z=XY$.
Thus assume that $(a,b,q)\neq(4,1,2)$. Consequently, $H=\Sp_a(q^b)'=\Sp_a(q^b)$. Since
\[
H\cap K=H\cap G_v=H_v=[(q^b)^{a-1}]{:}\Sp_{a-2}(q^b)=[q^{n-b}]{:}\Sp_{a-2}(q^b),
\]
we obtain
\[
\frac{|H|}{|H\cap K|}=\frac{|\Sp_a(q^b)|}{|q^{n-b}{:}\Sp_{a-2}(q^b)|}=q^n-1=\frac{|\SL_n(q)|}{|q^{n-1}{:}\SL_{n-1}(q)|}=\frac{|G|}{|K|},
\]
which implies $G=HK$. Hence $Z=\overline{G}=\overline{H}\,\overline{K}=XY$.
\end{proof}

Now we construct $(X,Y)$ in Row~2 of Table~\ref{TabLinear}, which will be displayed in Lemma~\ref{LemLinear04}. The next lemma will also be needed for symplectic groups.

\begin{lemma}\label{LemLinear03}
Let $G=\Sp_6(q)$ with $q$ even, let $H=\G_2(q)'<G$, and let $K=q^5{:}\Sp_4(q)$ be the subgroup of $G$ stabilizing a nonzero vector.
Then $G=HK$ with 
\[
H\cap K=[(q^5,q^6/4)]{:}\SL_2(q)=
\begin{cases}
2^{2+2}{:}\SL_2(2)&\textup{if }q=2\\
q^{2+3}{:}\SL_2(q)&\textup{if }q\geqslant4.
\end{cases}
\]
\end{lemma}

\begin{proof}
For $q=2$, we have $H=\G_2(2)'\cong\PSU_3(3)$, and computation in \magma~\cite{BCP1997} shows that $G=HK$ with $H\cap K=2^{2+2}{:}\SL_2(2)$.
Thus assume that $q\geqslant4$.
Consequently, $H=\G_2(q)'=\G_2(q)$. 
Then from~\cite[4.3.7]{Wilson2009} we see that $H\cap K=q^{2+3}{:}\SL_2(q)$. Hence
\[
\frac{|G|}{|K|}=\frac{|\Sp_6(q)|}{|q^5{:}\Sp_4(q)|}=q^6-1=\frac{|\G_2(q)|}{|q^{2+3}{:}\SL_2(q)|}=\frac{|H|}{|H\cap K|},
\]
and so $G=HK$.
\end{proof}

\begin{lemma}\label{LemLinear04}
Let $G=\SL(V)=\SL_n(q)$ with $n=6b$ and $q$ even, let $H=\G_2(q^b)'<\Sp_6(q^b)<G$, let $K=G_v$, let $Z=\overline{G}$, let $X=\overline{H}$, and let $Y=\overline{K}$. 
Then 
\[
H\cap K=[(q^{5b},q^{6b}/4)]{:}\SL_2(q^b)=
\begin{cases}
2^{2+2}{:}\SL_2(2)&\textup{if }(n,q)=(6,2)\\
q^{2b+3b}{:}\SL_2(q^b)&\textup{if }(n,q)\neq(6,2),
\end{cases}
\]
and $Z=XY$ with $Z=\PSL_n(q)$, $X=\lefthat\G_2(q^b)'$ and $Y\cong K=q^{n-1}{:}\SL_{n-1}(q)$.
\end{lemma}

\begin{proof}
It is clear that $Z=\PSL_n(q)$, $X=\lefthat\G_2(q^b)'$, and $Y\cong K=q^{n-1}{:}\SL_{n-1}(q)$. 
By Lemmas~\ref{LemLinear02} and~\ref{LemLinear03} we have
\begin{align*}
H\cap K=H\cap(\Sp_6(q^b)\cap K)&=H\cap(q^{5b}{:}\Sp_4(q^b))\\
&=[(q^{5b},q^{6b}/4)]{:}\SL_2(q^b)=
\begin{cases}
2^{2+2}{:}\SL_2(2)&\textup{if }(n,q)=(6,2)\\
q^{2b+3b}{:}\SL_2(q^b)&\textup{if }(n,q)\neq(6,2).
\end{cases}
\end{align*}
Hence
\[
\frac{|G|}{|K|}=\frac{|\SL_n(q)|}{|q^{n-1}{:}\SL_{n-1}(q)|}=q^n-1=q^{6b}-1=\frac{|\G_2(q^b)'|}{|[(q^{5b},q^{6b}/4)]{:}\SL_2(q^b)|}=\frac{|H|}{|H\cap K|},
\]
and so $G=HK$, which implies $Z=\overline{G}=\overline{H}\,\overline{K}=XY$.
\end{proof}

The factor pair $(X,Y)$ in Row~3 of Table~\ref{TabLinear} is constructed in the following lemma.

\begin{lemma}\label{LemLinear05}
Let $G=\SL(V)=\SL_n(q)$ with $n$ even, let $H=\Sp_n(q)'<G$, let $K=G_{v,W}$, let $Z=\overline{G}$, let $X=\overline{H}$, and let $Y=\overline{K}$. 
Then $H\cap K=\Sp_{n-2}(q)$, and $Z=XY$ with $Z=\PSL_n(q)$, $X=\lefthat\Sp_n(q)'$ and $Y\cong K=\SL_{n-1}(q)$.
\end{lemma}

\begin{proof}
It is clear that $Z=\PSL_n(q)$, $X=\lefthat\Sp_n(q)'$, and $Y\cong K=\SL_{n-1}(q)$. 
For $(n,q)=(4,2)$, we have $H=\Sp_4(2)'\cong\A_6$, and computation in \magma~\cite{BCP1997} shows that $Z=XY$.
Thus assume that $(n,q)\neq(4,2)$.
Consequently, $H=\Sp_n(q)'=\Sp_n(q)$. 
Let $\beta$ be a nondegenerate alternating form on $V$ with standard basis $e_1,f_1,\dots,e_{n/2},f_{n/2}$.
Without loss of generality, assume that $H=\Sp(V,\beta)$, $v=e_1$, and $W=\langle f_1,e_2,f_2,\dots,e_{n/2},f_{n/2}\rangle$. Then
\[
H\cap K=H\cap G_{v,W}=H\cap G_{e_1,W}=H_{e_1,W}=\Sp_{n-2}(q),
\]
and so
\[
\frac{|G|}{|K|}=\frac{|\SL_n(q)|}{|\SL_{n-1}(q)|}=q^{n-1}(q^n-1)=\frac{|\Sp_n(q)|}{|\Sp_{n-2}(q)|}=\frac{|H|}{|H\cap K|}.
\]
It follows that $G=HK$ and hence $Z=\overline{G}=\overline{H}\,\overline{K}=XY$.
\end{proof}

The factor pairs $(X,Y)$ in Rows~4--7 of Table~\ref{TabLinear} are constructed in the following four lemmas.

\begin{lemma}\label{LemLinear06}
Let $G=\SiL(V)=\SiL_n(q)$ with $q\in\{2,4\}$ and $n=2m$, let $H=\SiL(V_\sharp)<G$, let $K=G_{v,W}$, let $Z=\overline{G}$, let $X=\overline{H}$, and let $Y=\overline{K}$. Then 
\[
H\cap K=H^{(\infty)}\cap K^{(\infty)}=\SL_{m-1}(q^2), 
\]
and $Z=XY$ with $Z=\PSiL_n(q)$, $X=\SiL_m(q^2)/(m,q-1)$ and $Y\cong K=\SiL_{n-1}(q)$.
\end{lemma}

\begin{proof}
Without loss of generality, assume that $v=v_1$ and $W$ is the subspace of $V$ spanned by $\lambda v_1,v_2,\lambda v_2,\dots,v_m,\lambda v_m$.
It is clear that $Z=\PSiL_n(q)$, $X=\SiL_m(q^2)/(m,q-1)$, and $Y\cong K=\SiL_{n-1}(q)$.
Let $S=H^{(\infty)}$. Then $S=\SL(V_\sharp)=\SL_m(q^2)$, and $H=S{:}\langle\psi\rangle$.

We first calculate $S\cap K$. 
Let $U$ be the subspace of $V_\sharp$ spanned by $v_2,\dots,v_m$, and let $g\in S\cap K$.
Since $g\in S=\SL(V_\sharp)$, it follows that $U^g$ is a hyperplane of $V_\sharp$.
Since $g\in K=G_{v,W}\leqslant G_W$, we have $U^g\subseteq W$.
Then as $U$ is the only hyperplane of $V_\sharp$ that is contained in $W$, we conclude that $U^g=U$.
This together with $g\in K=G_{v,W}\leqslant G_v$ implies that $g\in S_{v,U}$.
Conversely, each element of $S_{v,U}$ lies in $G_{v,W}$ as it stabilizes $\langle\lambda v,U\rangle_{\bbF_q}=W$.
Hence $S\cap K=S_{v,U}=\SL_{m-1}(q^2)$. As $S\cap K^{(\infty)}$ is a normal subgroup of $S\cap K$ of index at most $2$, this implies that
$S\cap K^{(\infty)}=H^{(\infty)}\cap K^{(\infty)}=\SL_{m-1}(q^2)=S\cap K$.

Now as $q\in\{2,4\}$ we have $q=2f$, and so
\[
\frac{|G|}{|K|}=\frac{|\SiL_n(q)|}{|\SiL_{n-1}(q)|}=q^{n-1}(q^n-1)=2fq^{2m-2}(q^{2m}-1)=\frac{|\SiL_m(q^2)|}{|\SL_{m-1}(q^2)|}=\frac{|H|}{|S\cap K|}.
\]
Thus it suffices to prove $H\cap K=S\cap K$, or equivalently, $(H\setminus S)\cap K=\emptyset$. 
Suppose for a contradiction that there exists $k\in(H\setminus S)\cap K$. 
Since $\psi$ has order $2f$, we have 
\[
H=S{:}\langle\psi\rangle=S\cup\psi S\cup\dots\cup\psi^{2f-1}S.
\]
If $k\in\psi^fS$, then $k=\psi^fs$ for some $s\in S$.
If $k\notin\psi^fS$, then $q=4$ and $k^2\in\psi^fS$, which means that $k^2=\psi^fs$ for some $s\in S$.
In either case, there exists $s\in S$ such that $\psi^fs\in K$.
Since $K\leqslant G_v=G_{v_1}$ and $s\in S=\SL(V_\sharp)$, it follows that $v_1^s=v_1^{\psi^fs}=v_1$ and
\[
(\lambda v_1)^{\psi^fs}=(\lambda^{2^f}v_1)^s=\lambda^{2^f}v_1^s=\lambda^{2^f}v_1=\lambda^qv_1.
\]
Write $\lambda^q=a\lambda+b$ with $a,b\in\bbF_q$. Then $a$ and $b$ are both nonzero as $\lambda$ is a generator of $\bbF_{q^2}$.
Hence $(\lambda v_1)^{\psi^fs}=\lambda^qv_1=a\lambda v_1+bv_1\notin W$, contradicting the condition that $\psi^fs\in K\leqslant G_W$.
\end{proof}

\begin{lemma}\label{LemLinear16}
Let $G=\SL(V){:}\langle\phi\gamma\rangle$ with $q\in\{2,4\}$ and $n=2m$, let $H=\SL(V_\sharp){:}\langle\psi\gamma\rangle<G$, let $K=\SL_{n-1}(q){:}\langle\phi\gamma\rangle<G$, let $Z=\overline{G}$, let $X=\overline{H}$, and let $Y=\overline{K}$. 
Then $H\cap K=\SL_{m-1}(q^2)$, and $Z=XY$ with $Z=\PSL_n(q).2$, $X=(\SL_m(q^2).(2f))/(m,q-1)$ and $Y\cong K=\SL_{n-1}(q).2$.
\end{lemma}

\begin{proof}
Without loss of generality, assume $K^{(\infty)}=\SL(V)_{v,W}$ such that $v=v_1$ and $W$ is the image of $\langle v\rangle_{\bbF_q}$ under the action of $\gamma$. 
It is clear that $Z=\PSL_n(q).2$, $X=(\SL_m(q^2).(2f))/(m,q-1)$ and $Y\cong K=\SL_{n-1}(q).2$.
Let $g\in H\cap K$. Then 
\begin{equation}\label{EqnLinear02}
g=s\psi^i\gamma^i=t\phi^j\gamma^j
\end{equation}
for some $s\in\SL(V_\sharp)$, $t\in\SL(V)_{v,W}$, $i\in\{0,1,\dots,2f-1\}$ and $j\in\{0,1\}$.
From~\eqref{EqnLinear02} we deduce that $\GaL(V)\gamma^i=\GaL(V)g=\GaL(V)\gamma^j$ and hence $\gamma^i=\gamma^j$, which leads to $s\psi^i=t\phi^j$.
Since $s\psi^i=t\phi^j$ lies in $\SiL(V_\sharp)\cap\SiL(V)_{v,W}$ and Lemma~\ref{LemLinear06} shows that 
\[
\SiL(V_\sharp)\cap\SiL(V)_{v,W}=\SL(V_\sharp)\cap\SL(V)_{v,W}=\SL_{m-1}(q^2),
\]
it follows that $i=j=0$, and then~\eqref{EqnLinear02} gives $g=s=t\in\SL(V_\sharp)\cap\SL(V)_{v,W}=H^{(\infty)}\cap K^{(\infty)}$.
Since $g$ is arbitrary in $H\cap K$, it follows that 
\[
H\cap K=H^{(\infty)}\cap K^{(\infty)}=\SL_{m-1}(q^2).
\]
As $q\in\{2,4\}$ we have $q=2f$, and so
\[
\frac{|G|}{|K|}=\frac{|\SL_n(q){:}2|}{|\SL_{n-1}(q){:}2|}=q^{n-1}(q^n-1)=2fq^{2m-2}(q^{2m}-1)=\frac{|\SL_m(q^2){:}(2f)|}{|\SL_{m-1}(q^2)|}=\frac{|H|}{|H\cap K|}.
\]
Hence $G=HK$, which implies that $Z=\overline{G}=\overline{H}\,\overline{K}=XY$.
\end{proof}

\begin{lemma}\label{LemLinear07}
Let $G=\SiL(V)=\SiL_n(q)$ with $q\in\{2,4\}$ and $n=2m$ for some even $m$, let $H=\GaSp_m(q^2)<\SiL_m(q^2)<G$, let $K=G_{v,W}$, let $Z=\overline{G}$, let $X=\overline{H}$, and let $Y=\overline{K}$. Then $H\cap K=\Sp_{m-2}(q^2)$, and $Z=XY$ with $Z=\PSiL_n(q)$, $X=\GaSp_m(q^2)$ and $Y\cong K=\SiL_{n-1}(q)$.
\end{lemma}

\begin{proof}
It is clear that $Z=\PSiL_n(q)$, $X\cong H=\GaSp_m(q^2)$, and $Y\cong K=\SiL_{n-1}(q)$. 
By Lemmas~\ref{LemLinear06} and~\ref{LemLinear05} we have
\[
H\cap K=H\cap(\SiL_m(q^2)\cap K)=H\cap\SL_{m-1}(q^2)=\Sp_{m-2}(q^2).
\]
Observe that $q=2f$ as $q\in\{2,4\}$. It follows that
\[
\frac{|G|}{|K|}=\frac{|\SiL_n(q)|}{|\SiL_{n-1}(q)|}=q^{n-1}(q^n-1)=2fq^{2m-2}(q^{2m}-1)=\frac{|\GaSp_m(q^2)|}{|\Sp_{m-2}(q^2)|}=\frac{|H|}{|H\cap K|}.
\]
This implies $G=HK$ and hence $Z=\overline{G}=\overline{H}\,\overline{K}=XY$.
\end{proof}

\begin{lemma}\label{LemLinear17}
Let $G=\SL(V){:}\langle\phi\gamma\rangle$ with $q\in\{2,4\}$ and $n=2m$ for some even $m$, let $H=\Sp_m(q^2){:}\langle\psi\gamma\rangle<\SL_m(q^2){:}\langle\psi\gamma\rangle<G$, let $K=\SL_{n-1}(q){:}\langle\phi\gamma\rangle<G$, let $Z=\overline{G}$, let $X=\overline{H}$, and let $Y=\overline{K}$. 
Then $H\cap K=\Sp_{m-2}(q^2)$, and $Z=XY$ with $Z=\PSL_n(q).2$, $X=\Sp_m(q^2).(2f)$ and $Y\cong K=\SL_{n-1}(q).2$.
\end{lemma}

\begin{proof}
It is clear that $Z=\PSL_n(q).2$, $X=\Sp_m(q^2).(2f)$ and $Y\cong K=\SL_{n-1}(q).2$.
By Lemmas~\ref{LemLinear16} and~\ref{LemLinear05} we have
\[
H\cap K=H\cap((\SL_m(q^2){:}\langle\psi\gamma\rangle)\cap K)=H\cap\SL_{m-1}(q^2)=\Sp_{m-2}(q^2).
\]
Then similarly as in the proof of Lemma~\ref{LemLinear07} we obtain $Z=XY$.
\end{proof}

Finally we construct $(X,Y)$ in Row~8 of Table~\ref{TabLinear}, which will be displayed in Lemma~\ref{LemLinear09}. The next lemma will also be needed for symplectic groups.

\begin{lemma}\label{LemLinear08}
Let $G=\Sp_6(q)$ with $q$ even, let $H=\G_2(q)<G$, and let $K=\Sp_4(q)<\N_2[G]$. Then $G=HK$ with $H\cap K=\SL_2(q)$.
\end{lemma}

\begin{proof}
From~\cite[4.3.6]{Wilson2009} we see that $H\cap K=\SL_2(q)$. Hence
\[
\frac{|G|}{|K|}=\frac{|\Sp_6(q)|}{|\Sp_4(q)|}=q^5(q^6-1)=\frac{|\G_2(q)|}{|\SL_2(q)|}=\frac{|H|}{|H\cap K|},
\]
and so $G=HK$.
\end{proof}

\begin{remark}
If we let $H=\G_2(q)'$ in Lemma~\ref{LemLinear08}, then computation in \magma~\cite{BCP1997} shows that the conclusion $G=HK$ would not hold for $q=2$. 
\end{remark}

\begin{lemma}\label{LemLinear09}
Let $G=\SL(V)=\SL_6(q)$ with $n=6$ and $q$ even, let $H=\G_2(q)<\Sp_6(q)<G$, let $K=G_{v,W}$, let $Z=\overline{G}$, let $X=\overline{H}$, and let $Y=\overline{K}$. 
Then $H\cap K=\SL_2(q)$, and $Z=XY$ with $Z=\PSL_6(q)$, $X=\G_2(q)$ and $Y\cong K=\SL_5(q)$.
\end{lemma}

\begin{proof}
It is clear that $Z=\PSL_6(q)$, $X=\G_2(q)$, and $Y\cong K=\SL_5(q)$. By Lemmas~\ref{LemLinear05} and~\ref{LemLinear08} we have
\[
H\cap K=H\cap(\Sp_6(q)\cap K)=H\cap\Sp_4(q)=\SL_2(q).
\]
Hence
\[
\frac{|G|}{|K|}=\frac{|\SL_6(q)|}{|\SL_5(q)|}=q^5(q^6-1)=\frac{|\G_2(q)|}{|\SL_2(q)|}=\frac{|H|}{|H\cap K|},
\]
and so $G=HK$, which implies that $Z=\overline{G}=\overline{H}\,\overline{K}=XY$.
\end{proof}

\begin{remark}
If we let $H=\G_2(q)'$ in Lemma~\ref{LemLinear09} then the conclusion $Z=XY$ would not hold for $q=2$. 
\end{remark}

\section{Sporadic cases of $(X,Y)$ in Table~\ref{TabLinear}}\label{SecLinear02}

The factor pairs $(X,Y)$ in Rows~9--13 of Table~\ref{TabLinear} are constructed in Lemmas~\ref{LemLinear10}--\ref{LemLinear14} below, which are verified by computation in \magma~\cite{BCP1997}.

\begin{lemma}\label{LemLinear10}
Let $Z=\PSL_2(9)\cong\A_6$. Then $Z$ has precisely two conjugacy classes of subgroups isomorphic to $\A_5$. Let $X$ and $Y$ be two subgroups from these two classes respectively. Then $Z=XY$ with $X\cap Y=\D_{10}$.
\end{lemma}

The maximal subgroups of almost simple groups with socle $\PSL_3(4)$ can be found in ~\cite{CCNPW1985}.

\begin{lemma}\label{LemLinear11}
Let $L=\PSL_3(4)$, let $Z=L.2$ be an almost simple group with socle $L$ such that $Z$ has three conjugacy classes of maximal subgroups isomorphic to $\PGL_2(7)$, and let $X$ be such a maximal subgroup of $Z$. Then there are precisely two conjugacy classes of maximal subgroups $Y$ of $Z$ isomorphic to $\M_{10}$ such that $Z=XY$. For each such pair $(X,Y)$ we have $X\cap Y=\Sy_3$.
\end{lemma}

The maximal subgroups of $\SL_4(2)$ can be found in ~\cite{CCNPW1985}.

\begin{lemma}\label{LemLinear12}
Let $Z=\SL(V)=\SL_4(2)$ with $(n,q)=(4,2)$, let $X=Z_{v,W}$ or $Z_v$, and let $Y=\A_7$ be a maximal subgroup of $Z$. Then $Z=XY$ with 
\[
X\cap Y=
\begin{cases}
7{:}3&\textup{if }X=Z_{v,W}\\
\PSL_2(7)&\textup{if }X=Z_v.
\end{cases}
\]
\end{lemma}

\begin{lemma}\label{LemLinear13}
Let $G=\SL(V)=\SL_4(3)$ with $(n,q)=(4,3)$, let $K=G_v$, let $Z=\overline{G}$, and let $Y=\overline{K}$. Then $Z$ has precisely two (out of four) conjugacy classes of subgroups $X$ isomorphic to $\Sy_5$ such that $Z=XY$, while each subgroup $X$ of $Z$ of the form $4\times\A_5$ or $2^4{:}\A_5$ satisfies $Z=XY$. For each such pair $(X,Y)$ we have $Y\cong K=3^3{:}\SL_3(3)$ and
\[
X\cap Y=
\begin{cases}
3&\textup{if }X=\Sy_5\\
\Sy_3&\textup{if }X=4\times\A_5\\
\SL_2(3)&\textup{if }X=2^4{:}\A_5.
\end{cases}
\]
\end{lemma}

\begin{lemma}\label{LemLinear14}
Let $G=\SL(V)=\SL_6(3)$ with $(n,q)=(6,3)$, let $K=G_v$, let $Z=\overline{G}$, let $X=\PSL_2(13)$ be a subgroup of $Z$ (there are two conjugacy classes of such subgroups in $Z$), and let $Y=\overline{K}$. Then $Z=XY$ with $Z=\PSL_6(3)$, $Y=5^3{:}\SL_5(3)$ and $X\cap Y=3$.
\end{lemma}

In the following lemma we construct the factor pairs $(X,Y)$ in Rows~14 and~15 of Table~\ref{TabLinear}.

\begin{lemma}\label{LemLinear15}
Let $G=\SiL(V)=\SiL_{12}(q)$ with $n=12$ and $q\in\{2,4\}$, let $H=\GaG_2(q^2)=\Aut(\G_2(q^2))<\SiL_6(q^2)<G$, let $K=G_{v,W}$, let $Z=\overline{G}$, let $X=\overline{H}$, and let $Y=\overline{K}$. 
Then $H\cap K=\SL_2(q^2)$, and $Z=XY$ with $Z=\PSiL_{12}(q)$, $X=\GaG_2(q^2)$ and $Y\cong K=\SiL_{11}(q)$.
\end{lemma}

\begin{proof}
It is clear that $Z=\PSL_6(q)$, $X=\GaG_2(q^2)$ and $Y\cong K=\SiL_{11}(q)$.
By Lemmas~\ref{LemLinear06} and~\ref{LemLinear09} we have
\[
H\cap K=H\cap(\SiL_6(q^2)\cap K)=H\cap\SL_5(q^2)=\SL_2(q^2).
\]
Observe that $q=2f$ as $q\in\{2,4\}$. Thus
\[
\frac{|G|}{|K|}=\frac{|\SiL_{12}(q)|}{|\SiL_{11}(q)|}=q^{11}(q^{12}-1)=2fq^{10}(q^{12}-1)=\frac{|\GaG_2(q^2)|}{|\SL_2(q^2)|}=\frac{|H|}{|H\cap K|},
\]
and so $G=HK$, which implies $Z=\overline{G}=\overline{H}\,\overline{K}=XY$.
\end{proof}

\begin{lemma}\label{LemLinear18}
Let $G=\SL(V){:}\langle\phi\gamma\rangle$ with $n=12$ and $q\in\{2,4\}$, let $H=\G_2(q^2){:}\langle\psi\gamma\rangle<\Sp_6(q^2){:}\langle\psi\gamma\rangle<G$, let $K=\SL_{n-1}(q){:}\langle\phi\gamma\rangle<G$, let $Z=\overline{G}$, let $X=\overline{H}$, and let $Y=\overline{K}$. 
Then $H\cap K=\SL_2(q^2)$, and $Z=XY$ with $Z=\PSL_{12}(q).2$, $X=\G_2(q^2).(2f)\cong\GaG_2(q^2)$ and $Y\cong K=\SL_{11}(q).2$.
\end{lemma}

\begin{proof}
It is clear that $Z=\PSL_{12}(q).2$, $X=\G_2(q^2).(2f)\cong\GaG_2(q^2)$ and $Y\cong K=\SL_{11}(q).2$.
By Lemmas~\ref{LemLinear16} and~\ref{LemLinear09} we have
\[
H\cap K=H\cap((\SL_6(q^2){:}\langle\psi\gamma\rangle)\cap K)=H\cap\SL_5(q^2)=\SL_2(q^2).
\]
Then similarly as in the proof of Lemma~\ref{LemLinear15} we obtain $Z=XY$.
\end{proof}

\section{Proof of Theorem~\ref{ThmLinear}}

Let $G$ be an almost simple group with socle $L=\PSL_n(q)$, and let $H$ and $K$ be nonsolvable subgroups of $G$ not containing $L$.
In Sections~\ref{SecLinear01} and~\ref{SecLinear02} it is shown that all pairs $(X,Y)$ in Table~\ref{TabLinear} are factor pairs of $L$.
Hence by Lemma~\ref{LemXia02} we only need to prove that, if $G=HK$, then $(H,K)$ tightly contains $(X^\alpha,Y^\alpha)$ for some $(X,Y)$ in Table~\ref{TabLinear} and $\alpha\in\Aut(L)$.
Suppose that $G=HK$. Then by~\cite[Theorem~3.1]{LWX} the triple $(L,H^{(\infty)},K^{(\infty)})$ lies in Table~\ref{TabInftyLinear}.

\begin{table}[htbp]
\captionsetup{justification=centering}
\caption{$(L,H^{(\infty)},K^{(\infty)})$ for linear groups}\label{TabInftyLinear}
\begin{tabular}{|l|l|l|l|l|l|}
\hline
Row & $L$ & $H^{(\infty)}$ & $K^{(\infty)}$ & Conditions\\
\hline
1 & $\PSL_n(q)$ & $\lefthat\SL_a(q^b)$, $\lefthat\Sp_a(q^b)'$ & $q^{n-1}{:}\SL_{n-1}(q)$ & $n=ab$\\
2 & $\PSL_n(q)$ & $\G_2(q^b)'$ & $q^{n-1}{:}\SL_{n-1}(q)$ & $n=6b$, $q$ even\\
\hline
3 & $\PSL_n(q)$ & $\lefthat\Sp_n(q)'$ & $\SL_{n-1}(q)$ & \\
4 & $\SL_n(2)$ & $\SL_{n/2}(4)$, $\Sp_{n/2}(4)$ & $\SL_{n-1}(2)$ & \\
5 & $\PSL_n(4)$ & $\lefthat\SL_{n/2}(16)$, $\Sp_{n/2}(16)$ & $\SL_{n-1}(4)$ & \\
6 & $\PSL_6(q)$ & $\G_2(q)'$ & $\SL_5(q)$ & $q$ even\\
\hline
7 & $\PSL_2(9)$ & $\PSL_2(5)$ & $\A_5$ & \\
8 & $\PSL_3(4)$ & $\PSL_2(7)$ & $\A_6$ & \\
9 & $\PSL_4(2)$ & $\SL_3(2)$, $2^3{:}\SL_3(2)$ & $\A_7$ & \\
10 & $\PSL_4(3)$ & $\A_5$, $2^4{:}\A_5$ & $3^3{:}\SL_3(3)$ & \\
11 & $\PSL_6(3)$ & $\PSL_2(13)$ & $3^5{:}\SL_5(3)$ & \\
12 & $\SL_{12}(2)$ & $\G_2(4)$ & $\SL_{11}(2)$ & \\
13 & $\PSL_{12}(4)$ & $\G_2(16)$ & $\SL_{11}(4)$ & \\
\hline
\end{tabular}
\vspace{3mm}
\end{table}

If $(L,H^{(\infty)},K^{(\infty)})$ lies in Rows~1--3 of Table~\ref{TabInftyLinear}, then $(H,K)$ tightly contains $(X^\alpha,Y^\alpha)$ for some pair $(X,Y)$ in Rows~1--3 of Table~\ref{TabLinear} and $\alpha\in\Aut(L)$.

For $(L,H^{(\infty)},K^{(\infty)})$ in Row~6 of Table~\ref{TabInftyLinear}, viewing the remark after Lemma~\ref{LemLinear09}, we see that $(H,K)$ tightly contains the pair $(X,Y)=(\G_2(q),\SL_5(q))$ in Row~8 of Table~\ref{TabLinear}.

If $(L,H^{(\infty)},K^{(\infty)})$ lies in Rows~7--11 of Table~\ref{TabInftyLinear}, then computation in \magma~\cite{BCP1997} shows that $(H,K)$ tightly contains $(X^\alpha,Y^\alpha)$ for some pair $(X,Y)$ in Rows~9--13 of Table~\ref{TabLinear} and $\alpha\in\Aut(L)$.

Finally, if $(L,H^{(\infty)},K^{(\infty)})$ lies in Rows~4--5 or~12--13 of Table~\ref{TabInftyLinear}, then the following lemma shows that $(H,K)$ tightly contains $(X^\alpha,Y^\alpha)$ for some pair $(X,Y)$ in Rows~4--7 or~14--15 of Table~\ref{TabLinear} and $\alpha\in\Aut(L)$.

\begin{lemma}
Suppose that $K^{(\infty)}=\SL_{n-1}(q)$ with $q\in\{2,4\}$, and either $H^{(\infty)}=\SL_{n/2}(q^2)$ or $\Sp_{n/2}(q^2)$, or $H^{(\infty)}=\G_2(q^2)$ with $n=12$. Then $(H,K)$ tightly contains $(X^\alpha,Y^\alpha)$ for some pair $(X,Y)$ in Rows~\emph{4--7} or~\emph{14--15} of Table~$\ref{TabLinear}$ and $\alpha\in\Aut(L)$.
\end{lemma}

\begin{proof}
First suppose that $H\leqslant\PGL(V_\sharp).(\langle\psi^2\rangle\times\langle\gamma\rangle)$.
Since $K$ is contained in an antiflag stabilizer of $G$, we derive from~$G=HK$ that $H$ is antiflag-transitive, and so is $\PGL(V_\sharp).(\langle\psi^2\rangle\times\langle\gamma\rangle)$.
Consequently, $\PGL(V_\sharp).\langle\psi^2\rangle$ is antiflag-transitive.
In particular, there exist $h\in\GL(V_\sharp)$ and $i\in\{0,1\}$ such that $\psi^{2i}h$ sends the antiflag $\{\langle v\rangle_{\bbF_q},\langle\lambda v,U\rangle_{\bbF_q}\}$ to $\{\langle\lambda^2v\rangle_{\bbF_q},\langle v,U\rangle_{\bbF_q}\}$, where $v=v_1$ and $U=\langle v_2,\lambda v_2,\dots,v_{n/2},\lambda v_{n/2}\rangle_{\bbF_q}$.
This implies that $v^{\psi^{2i}h}\in\langle\lambda^2v\rangle_{\bbF_q}$ and $(\lambda v)^{\psi^{2i}h}\in\langle v,U\rangle_{\bbF_q}$.
Hence $v^h=v^{\psi^{2i}h}=\mu\lambda^2v$ for some $\mu\in\bbF_q$, and then
\[
\mu\lambda^{qi+2}v=\lambda^{qi}(\mu\lambda^2v)=\lambda^{qi}v^h=(\lambda v)^{\psi^{2i}h}\in\langle v,U\rangle_{\bbF_q},
\] 
which leads to $\lambda^{qi+2}\in\bbF_q$, a contradiction.

Therefore, up to a conjugation of some $\alpha\in\Aut(L)$ on $H$ and $K$ at the same time, we have $H\geqslant H^{(\infty)}.\langle\psi\rangle$ or $H^{(\infty)}.\langle\psi\gamma\rangle$. 
Applying this conclusion to the factorization $HL\cap KL=(H\cap KL)(K\cap HL)$ we obtain that $H\cap KL\geqslant H^{(\infty)}.\langle\psi\rangle$ or $H^{(\infty)}.\langle\psi\gamma\rangle$.
Hence either $H\geqslant H^{(\infty)}.\langle\psi\rangle$ and $K\geqslant K^{(\infty)}.\langle\phi\rangle$, or $H\geqslant H^{(\infty)}.\langle\psi\gamma\rangle$ and $K\geqslant K^{(\infty)}.\langle\phi\gamma\rangle$.
This means that $(H,K)$ tightly contains the pair $(X,Y)$ in Rows~4--7 or~14--15 of Table~\ref{TabLinear}.
\end{proof}

\section*{Acknowledgments}
The first author acknowledges the support of NNSFC grants no.~11771200 and no.~11931005. The second author acknowledges the support of NNSFC grant no.~12061083.


\begin{thebibliography}{}

\bibitem{BCP1997}
W. Bosma, J. Cannon and C. Playoust,
The magma algebra system I: The user language,
\emph{J. Symbolic Comput.}, 24 (1997), no. 3-4, 235--265.

\bibitem{BHR2009}
J. N. Bray, D. F. Holt and C. M. Roney-Dougal,
Certain classical groups are not well-defined,
\emph{J. Group Theory}, 12 (2009), 171--180.

\bibitem{BHR2013}
J. N. Bray, D. F. Holt and C. M. Roney-Dougal,
\emph{The maximal subgroups of the low-dimensional finite classical groups}, Cambridge University Press, Cambridge, 2013.

\bibitem{BG2016}
T. Burness and M. Giudici, 
\emph{Classical Groups, Derangements and Primes}, Cambridge University Press, Cambridge, 2016.

\bibitem{BL}
T. C. Burness and C. H. Li,
On solvable factors of almost simple groups,
\emph{Adv. Math.}, 377 (2021), 107499, 36 pp.


\bibitem{CK1979}
P. J. Cameron and W. M. Kantor,
$2$-transitive and antiflag transitive collineation groups of finite projective spaces,
\emph{J. Algebra}, 60 (1979), no. 2, 384--422.


\bibitem{CCNPW1985}
J. H. Conway, R. T. Curtis, S. P. Norton, R. A. Parker and R. A. Wilson,
\emph{Atlas of finite groups: maximal subgroups and ordinary characters for simple groups}, Clarendon Press, Oxford, 1985.








\bibitem{Giudici2006}
M. Giudici,
Factorisations of sporadic simple groups,
\emph{J. Algebra}, 304 (2006), no. 1, 311--323.

\bibitem{GGP}
M. Giudici, S. P. Glasby and C. E. Praeger,
Subgroups of Classical Groups that are Transitive on Subspaces,
https://arxiv.org/abs/2012.07213.






\bibitem{HLS1987}
C. Hering, M. W. Liebeck and J. Saxl,
The factorizations of the finite exceptional groups of Lie type,
\emph{J. Algebra}, 106 (1987), no. 2, 517--527.





\bibitem{Kantor}
W. M. Kantor,
Antiflag transitive collineation groups,
https://arxiv.org/pdf/1806.02203.








\bibitem{LWX}
C. H. Li, L. Wang and B. Xia,
The exact factorizations of almost simple groups,
https://arxiv.org/abs/2012.09551.


\bibitem{LX2019}
C. H. Li and B. Xia,
Factorizations of almost simple groups with a factor having many nonsolvable composition factors,
\emph{J. Algebra}, 528 (2019), 439--473.

\bibitem{LX}
C. H. Li and B. Xia,
Factorizations of almost simple groups with a solvable factor, and Cayley graphs of solvable groups,
to appear in \emph{Mem. Amer. Math. Soc.}, https://arxiv.org/abs/1408.0350.

\bibitem{Liebeck1987}
M. W. Liebeck,
The affine permutation groups of rank three,
\emph{Proc. London Math. Soc. (3)}, 54 (1987), no. 3, 477--516.

\bibitem{LPS1990}
M. Liebeck, C. E. Praeger and J. Saxl,
The maximal factorizations of the finite simple groups and their automorphism groups,
\emph{Mem. Amer. Math. Soc.},  86  (1990),  no. 432.

\bibitem{LPS1996}
M. W. Liebeck, C. E. Praeger and J. Saxl, On factorizations of almost simple groups,
\emph{J. Algebra}, 185 (1996), no. 2, 409--419.

\bibitem{LPS2000}
M. W. Liebeck, C. E. Praeger and J. Saxl, Transitive subgroups of primitive permutation groups,
\emph{J. Algebra}, 234 (2000), no. 2, 291--361.

\bibitem{LPS2010}
M. W. Liebeck, C. E. Praeger and J. Saxl, Regular subgroups of primitive permutation groups,
\emph{Mem. Amer. Math. Soc.}, 203 (2010),  no. 952.



\bibitem{Wielandt1979}
H. Wielandt,
Zusammengesetzte Gruppen: H\"{o}lders Programm heute,
in \emph{The Santa Cruz Conference on Finite Groups (Univ. California, Santa Cruz, Calif., 1979)}, pp. 161--173.



\bibitem{Wilson2009}
R. Wilson,
\emph{The finite simple groups}, Springer, 2009.


\bibliographystyle{100}

\end{thebibliography}
\end{document}